\documentclass[10pt]{amsart}
\usepackage{amsmath,amssymb,amscd,mathrsfs,amsthm,amscd,inputenc,enumerate,amsfonts,pdfsync}
\usepackage[english]{babel}
\usepackage{color}%

\newcommand{\spec}{{\rm Spec}}

\newcommand{\zar}{{\rm Zar}}

\newcommand{\qspec}{{\rm QSpec}}
\newcommand{\qmax}{{\rm QMax}}

\newcommand{\ms}{\mathscr}

\newcommand{\mf}{\mathbf}
\newcommand{\ad}{{\rm Ad}}
\newcommand{\f}{\mathfrak}

\newcommand{\z}{{\ldots}}

\newcommand{\inssemistar}{\mathrm{SStar}}
\newcommand{\insfinss}{\mathrm{SStar}_f}
\newcommand{\insss}{\mathrm{(S)Star}_f}
\newcommand{\inssubmod}{{\overline{\mf F}}}
\DeclareMathOperator{\chius}{Ad}
\DeclareMathOperator{\Over}{Over} %
\newtheorem{teor}{Theorem}[section]
\newtheorem{cor}[teor]{Corollary}
\newtheorem{prop}[teor]{Proposition}
\newtheorem{lemma}[teor]{Lemma}

\theoremstyle{definition}
\newtheorem{defin}[teor]{Definition}
\newtheorem{ex}[teor]{ \textbf{Example}}
\newtheorem{ex1}{ \textbf{Example}}[subsection]
\newtheorem{oss}[teor]{Remark}

\begin{document}


\title{Some topological considerations on semistar operations}
\author{Carmelo Antonio Finocchiaro}
\address{Dipartimento di Matematica, Universit\`{a} degli Studi ``Roma Tre'', Largo San Leonardo Murialdo, 1, 00146 Rome, Italy}
\email{carmelo@mat.uniroma3.it}

\author{Dario Spirito}
\email{spirito@mat.uniroma3.it}

\begin{abstract}
We consider properties and applications of a new topology, called \emph{the Zariski topology}, on the space $\inssemistar(A)$ of all the semistar operations on an integral domain $A$. We prove that the set of all overrings of $A$, endowed with the classical Zariski topology, is homeomorphic to a subspace of $\inssemistar(A)$. The  topology on $\inssemistar(A)$ provides a general theory, through which we see several algebraic properties of semistar operation as very particular cases of our construction.  Moreover, we show that the subspace $\insfinss(A)$ of all the semistar operations of finite type on $A$ is a spectral space.
\end{abstract}
\maketitle

\section*{Introduction}
The notions of star operation, introduced by Krull in \cite{kr}, and that of a system of ideals, studied extensively by E. Noether, H. Pr\"ufer, and P. Lorentzen from the 1930s, play a central role in the study of the multiplicative structure of the ideals of a ring.   These notions represent a natural and abstract setting to study problems on factorization of ideals. For a recent contribution on this circle of ideas see, for instance,  \cite{ju}.

There are two main ways to generalize star operations: the first is through semiprime operations (see \cite{ep} and \cite{va2}; however, note that they were born independently from the star operation setting), while the second is through the use of \emph{semistar operations}, introduced by A. Okabe and R. Matsuda in \cite{ok-ma}. Semistar operations represent a very powerful tool for classifying domains according to the properties of their ideals, allowing more flexibility than ``classical'' star operations. Such operations are also closely related to the theory of Kronecker function rings. For a deeper insight on the recent developements on this topic, see \cite{fifolo2}, \cite{folo3}, \cite{folo2}, \cite{folo}, \cite{ok-ma2}, \cite{ok-ma3}. 

The main goal of this paper is to {use a topological approach to} extend results in the literature concerning algebraic properties of semistar operations. Some motivation for studying the multiplicative structure of the ideals of an integral domain from a topological point of view can be found in \cite{fifolo2}. In this paper the authors endowed the Riemann-Zariski space $\zar(A)$ (see \cite{zar-sam}) of all the valuation overrings of an integral domain $A$ with several topological structures (the Zariski, the constructible and the inverse topologies) and studied the interplay between the topological properties of a given subspace of $\zar(A)$ and the algebraic properties of the semistar operation determined by such a subspace. By using this approach, the authors investigated, from a topological point of view, the representations of an integrally closed domain as an intersection of valuation overrings.  

In the meanwhile, B. Olberding in \cite{ol} defined the Zariski topology on the space $\Over(A)$ of all overrings of an integral domain $A$ in such a way that both the set of localizations of $A$ (with the topology induced by the Zariski topology of $\spec(A)$) and $\zar(A)$ become subspaces of $\Over(A)$. The main idea of Section 2 is to consider the set $\inssemistar(A)$ of all semistar operations on $A$ and to endow it with a new Zariski topology, in such a way that $\Over(A)$ is identifiable canonically with a subspace of $\inssemistar(A)$ (Proposition \ref{embedding}). After giving the main properties of the Zariski topology on $\inssemistar(A)$, we relate the compactness of the subspaces of $\inssemistar(A)$ with the  finite type property of their infimum (with respect to the natural order on $\inssemistar(A)$). We show (Proposition \ref{compact-fintype}) that the infimum of a compact family of semistar operations of finite type is of finite type, answering a question !
 posed in \cite{proc}, and that the converse is true when each operation of the family is induced by a localization of $A$ or by a valuation ring (Corollary \ref{fin-spectral-car} and Proposition \ref{valuations}). We also conjecture that the converse is true for any family of overrings. Then, we specialize the study of the Zariski topology on the subspace $\insfinss(A)$ of all the semistar operations on $A$ of finite type and we show that this space is spectral (Theorem \ref{spectral}). The proof we give is not constructive; it would be interesting to know if there is a canonical way to find a ring whose prime spectrum is homeomorphic to $\insfinss(A)$.

In Section \ref{sect:funct}, we show that the Zariski topology on $\inssemistar(A)$  has a natural functorial property, in the sense that if $A\subseteq B$  is an extension of integral domains, there is a natural continuous map $\inssemistar(B)\longrightarrow\inssemistar(A)$, which is an embedding if $B$ has the same quotient field as $A$.

The last section of the paper is devoted to a deeper study of the semistar operations induced by  localizations of $A$, which are  usually called \emph{spectral} semistar operations. Their study has been motivated by the possibility of relating the properties of a semistar operation $\star$ with the properties of its stable closure $\tilde{\star}$, which is always a spectral operation of finite type. {We find} a necessary and sufficient condition for the equality of the stable closure of two  semistar operations  (Proposition \ref{weakly-eq}), and identify the stable closure of a semifinite semistar operation (the definition will be recalled later). Both of these results are topological in nature, using the inverse topology of the Zariski topology.

\section{Background material}
We begin with some definitions and preliminary results. In the following with the term \emph{ring} we always mean a commutative ring with identity. Let $A$ be a ring. In the following, the set $\spec(A)$ will be often (but not always) endowed with the Zariski topology, i.e. the topology whose closed sets are of the form
$$
V(\f a):=\{\f p\in \spec(A):\f p\supseteq \f a\}
$$
for any ideal $\f a$ of $A$. The sets of the form $D(f):=\spec(A)-V(fA)$ ($f\in A$) form a basis of open sets for the Zariski topology. \\\\
\subsection{Semistar operations}\label{sect:background-semistar}
 Let $A$ be an integral domain and $K$ be the quotient field of $A$. Any ring $B$ such that $A\subseteq B\subseteq K$ will be called an \emph{ overring of $A$}. We will use the following notation:
\begin{itemize}
\item $\mf f(A)$ is the set of all nonzero finitely generated  fractional ideals of $A$.
\item $\mf F(A)$ is the set of all nonzero fractional ideals of $A$.
\item $\overline{\mf F}(A)$ is the set of all nonzero $A$-submodules of $K$. 
\end{itemize}
Of course, we have $\mf f(A)\subseteq \mf F(A)\subseteq \overline{\mf F}(A)$. For any $F\in \overline{\mf F}(A), G\in \inssubmod(A)\cup\{0\}$, set
$$
(F:G):=\{x\in K:xG\subseteq F\}.
$$
In 1994 A. Okabe and R. Matsuda (see \cite{ok-ma}) introduced the notion of a semistar operation. A semistar operation on $A$ is a function \,\,\, $\star:\overline{\mf F}(A)\longrightarrow\overline{\mf F}(A)$, $F\mapsto F^\star$ satisfying the following axioms, for any nonzero element $k\in K$ and every $F,G\in \overline{\mf F}(A)$:
\begin{itemize}
\item $(kF)^\star=kF^\star$;
\item $F\subseteq G$ implies $F^\star\subseteq G^\star$;
\item $F\subseteq F^\star$;
\item $(F^\star)^\star=F^\star$.
\end{itemize}
If $\star$ is a semistar operation on $A$, a submodule $F\in \inssubmod(A)$ is called \emph{$\star$-closed} if $F=F^\star$. An integral ideal $\f a$ of $A$ is called a \emph{quasi-$\star$-ideal of $A$} if either $\f a=(0)$ or $\f a^\star\cap A=\f a$. It follows easily by definition that $\f a^\star\cap A$ is a quasi-$\star$-ideal of $A$.

If $A$ is $\star$-closed, then $\star$ is called a \emph{(semi)star operation}, while $\star|_{\mf F(A)}$ is called a \emph{star operation}. For background on star operations, see \cite{gi}.

The set $\inssemistar(A)$ of all the semistar operations on $A$ is endowed with a natural partial order, defined by $\star\leq \star'$ if and only if $F^\star\subseteq F^{\star'}$, for any $F\in \overline{\mf F}(A)$. It is well known that $\star\leq \star'$ if and only if every nonzero $\star'$-closed $A$-submodule of $K$ is also $\star$-closed. Moreover, if $\star\leq\star'$, then a quasi-$\star'$-ideal of $A$ is also a quasi-$\star$-ideal, but the last condition does not imply that $\star\leq\star'$.

A prime ideal of $A$ that is a quasi-$\star$-ideal is called a \emph{quasi-$\star$-prime ideal} of $A$. A maximal element in the set of all the proper quasi-$\star$-ideals is called \emph{quasi-$\star$-maximal ideal} of $A$. We will use the following notation:
\begin{itemize}
\item $\qspec^\star(A)$ is the set of all the quasi-$\star$-prime ideals of $A$.
\item $\qmax^\star(A)$ is the set of all the quasi-$\star$-maximal ideals of $A$. 
\end{itemize}
It is easy to see that $\qmax^\star(A)\subseteq \qspec^\star(A)$. If $\star$ is of finite type, then a straightforward application of Zorn's Lemma shows that each proper quasi-$\star$-ideal is contained in a quasi-$\star$-maximal ideal of $A$. 

\begin{ex1}\label{ex:background-semistar}
This example lists some useful specific semistar operations and several techniques to construct specific classes of semistar operations and new semistar operations from old ones.
\begin{enumerate}[(a)]
\item $d:=d_A$ denotes the identity semistar operation on $A$. 
\item If $B$ is an overring of $A$, {we} denote by $\star_{\{B\}}$ the semistar operation on $A$ defined by setting $F^{\star_{\{B\}}}:=FB$, for any $F\in \overline{\mf F}(A)$.
\item Let $\star$ be a semistar operation on $A$. We can associate a new semistar operation $\star_f$ on $A$ by setting $$F^{\star_f}:=\bigcup\{G^\star:G\in \mf f(A) \mbox{ and } G\subseteq F\},$$
for any $F\in \overline{\mf F}(A)$. We call the semistar operation $\star_f$ \emph{the finite type closure of} $\star$. We call $\star$ \emph{a semistar operation of finite type} if $\star=\star_f$. Note that $(\star_f)_f=\star_f$, and thus $\star_f$ is a semistar operation of finite type. In the following, we shall denote by $\insfinss(A)$ the set of the semistar operations of finite type on $A$. 
\item\label{1.1d} $v$ denotes the divisorial semistar operation on $A$, defined by $F^v:=(A:(A:F))$, for any $F\in \overline{\mf F}(A)$. The finite type closure of $v$ is usually denoted by $t$. Note that $t$ (resp. $t|_{\mf F(A)}$) is the biggest (semi)star operation (resp., star operation) of finite type \cite[Lemma 1.12]{pi1}. Furthermore, if $A$ is integrally closed, then it is a Pr\"ufer domain if and only if $t|_{\mf F(A)}$ is the identity \cite[Proposition 34.12]{gi}. 
\item\label{bg:infsup} Let $\mathcal S$ be a nonempty collection of semistar operations on $A$. Then $\bigwedge(\mathcal S)$ is the semistar operation on $A$ defined by setting 
$$F^{\bigwedge(\mathcal S)}:=\bigcap\{F^\star:\star\in \mathcal S\}\qquad \mbox{ for any } F\in \overline{\mf F}(A)$$ 
It is easy to see that $\bigwedge(\mathcal S)$ is the infimum of $\mathcal S$ in the partially ordered set $(\inssemistar(A),\leq)$. Moreover, the semistar operation 
$$
\bigvee(\mathcal S):=\bigwedge(\{\sigma\in \inssemistar(A):\sigma\geq \star, \mbox{ for any }\star\in \mathcal S\})
$$
is the supremum of $\mathcal S$ in the partially ordered set $(\inssemistar(A),\leq)$.
\item Let $Y$ be a nonempty collection of overrings of $A$. By (\ref{bg:infsup}), $\wedge_Y:=\bigwedge(\{\star_{\{B\}}:B\in Y\})$ is a semistar operation on $A$. In other words, the semistar operation $\wedge_Y$ is defined by setting
$$
F^{\wedge_Y}:=\bigcap\{FB:B\in Y\} \qquad\mbox{ for any }F\in \overline{\mf F}(A).
$$
\item If $Y$ is a collection of valuation overrings of $A$, then  $\wedge_Y$ is called \emph{a valutative semistar operation}. In particular, when $Y$ is the set of all the valuation overrings of $A$,  $\wedge_Y$ is called \emph{the $b$-semistar operation} (or \emph{integral closure}) on $A$. 
\item Let $X$ be a nonempty collection of prime ideals of $A$. The semistar operation $s_X:=\wedge_{\{A_{\f p}:\f p\in X\}}$ is called \emph{a spectral semistar operation}. 
\item  We say that a semistar operation is \emph{stable} if $(F\cap G)^\star= F^\star\cap G^\star$ for any $F,G\in \overline{\mf F}(A)$. Since localization commutes with finite intersections, any spectral semistar operation is stable.
\item Let $\star$ be a semistar operation on $A$. For any $F\in \overline{\mf F}(A)$ set 
$$
F^{\widetilde{\star}}:=\bigcup\{ (F:\f a):\f a \mbox{ is a finitely generated ideal of  }A \mbox{ and } \f a^\star=A^\star\}.
$$
Then $\widetilde{\star}$ is a stable and of finite type semistar operation on $A$, called \emph{the stable closure of $\star$}. By \cite[Corollary 3.5(2)]{folo}, we have $\qmax^{\star_f}(A)=\qmax^{\widetilde \star}(A)$, and $\widetilde\star=s_{\qmax^{\star_f}(A)}$. Thus, $\widetilde \star$ is spectral. More precisely, by \cite[Corollary 3.9(2)]{fohu},  the following conditions are equivalent:
\begin{enumerate}[(i)]
\item $\star=\widetilde\star$.
\item $\star$ is stable and of finite type.
\item $\star$ is spectral and of finite type. 
\end{enumerate}
\end{enumerate}
\end{ex1}

\subsection{Spectral spaces.}\label{inv-rem}
Let $X$ be a topological space and let $Y\subseteq X$. We will denote by $\ad(Y)$ the closure of $Y$. A topological space is called \emph{spectral} if it is homeomorphic to the prime spectrum of a ring, endowed with the Zariski topology. In \cite{ho}, the author shows that a topological space is spectral if and only if it is compact (i.e., every open cover has a finite subcover), admits a basis of open and compact subspaces that is closed under finite intersection, and every irreducible closed subset has a unique generic point (i.e., it is the closure of a unique point).

Let $X$ be a spectral space, $\mathscr O$ its topology. It is possible to endow $X$ with another classical topological structure by taking, as basis of closed sets, the collection of all the open and compact subspaces of $(X,\mathscr O)$. The topology obtained is called \emph{the inverse topology on $X$}. Denote by $X^{\rm inv}$ the set $X$, endowed with the inverse topology, and by $\ad^{\rm i}(Y)$ the closure of a subset $Y$ of $X$, with respect to the inverse topology. By \cite[Proposition 8]{ho}, $X^{\rm inv}$ is a spectral space and $\ad^{\rm i}(\{x\})\subseteq \ad^{\rm i}(\{y\})$ if and only if $\ad(\{y\})\subseteq \ad(\{x\})$ (this justifies the choice of the name of this topology). Moreover, for every subset $Y\subseteq X$, $\ad^{\rm i}(Y)$ is compact, with respect to the topology $\mathscr O$ \cite[Remark 2.2 and Proposition 2.6]{fifolo2}.

\section{The Zariski topology on the set of all semistar operations}\label{sect:Zar}
The main goal of this paper is to define and study a new topology on the set of all semistar operations on an integral domain and to investigate how the algebraic properties  and the topological properties of semistar operations are related. In the following, $A$ is an integral domain and $K$ is the quotient field of $A$. 

\begin{defin}
The \emph{Zariski topology on $\inssemistar(A)$} is the topology for which a subbasis of open sets is the collection of all the sets of the form $V_F:=V_F^{(A)}:=\{\star\in\inssemistar(A):1\in F^\star\}$, as $F$ ranges among the nonzero $A$-submodules of $K$. The \emph{Zariski topology on $\insfinss(A)$} is just the subspace topology of the Zariski topology on $\inssemistar(A)$. 
\end{defin}
The following remark will help to focus on the most basic properties of the Zariski topology on $\inssemistar(A)$. 
\begin{oss}\label{basic} We preserve the notation of the beginning of the present section. The following statements hold. 
\begin{enumerate}[\rm (a)]
\item The semistar operation $\wedge_{\{K\}}$ (which sends every nonzero submodule $H$ to the quotient field $K$ of $A$) is clearly the maximum of $(\inssemistar(A),\leq)$. Since each $V_F$ contains $\wedge_{\{K\}}$, so does every nonempty open set; hence $\wedge_{\{K\}}$ is a generic point of $\inssemistar(A)$. Moreover, since $\wedge_{\{K\}}$ is obviously of finite type, we infer that $\insfinss(A)$ is a dense subspace of $\inssemistar(A)$. 
\item The identity operation $d$ is contained in $V_F$ if and only if $1\in F$, and thus if and only if $V_F=\inssemistar(A)$. Hence every nonempty closed set contains $d$: therefore, if $\star\in\inssemistar(A)\setminus\{d\}$, then $\{\star\}$ is not closed. We will see that $\{d\}$ is closed in $\inssemistar(A)$ in the following Proposition \ref{chiusura}. 
\item The topology of $\inssemistar(A)$  is naturally linked to the order $\leq$, in the following sense. If $U\subseteq \inssemistar(A)$ is an open neighborhood of $\star$ and $\star'\geq \star$, then $\star'\in U$. As a matter of fact, by definition there are $A$-submodules  $F_1,\z,F_n$ of $K$ such that $\star\in \bigcap_{i=1}^nV_{F_i}\subseteq U$. Since $\star'\geq \star$, we have $1\in F_i^\star\subseteq F_i^{\star'}$, for $i=1,\z,n$, and thus $\star'\in \bigcap_{i=1}^nV_{F_i}\subseteq U$. 
\item\label{basic:ft} The Zariski topology of $\insfinss(A)$ is determined by the finitely generated fractional ideals of $A$, in the sense that the collection of the sets of the form $U_F:=V_F\cap \insfinss(A)$, where $F$ varies among the finitely generated fractional ideals of $A$, is a subbase. As a matter of fact, it suffices to note that, for any $A$-submodule $G$ of $K$, we have
$$
V_G\cap \insfinss(A)=\bigcup\{U_F: F\in \mf f(A),F\subseteq G\}.
$$
\end{enumerate}
\end{oss}

\begin{prop}\label{chiusura}
We preserve the notation of the beginning of the present section. Then, for any $\star\in \inssemistar(A)$, we have $$\ad(\{\star\})=\{\star'\in \inssemistar(A): \star'\leq \star\}.$$ In particular, $\{d\}$ is the unique closed point in $\inssemistar(A)$. 
\end{prop}
\begin{proof}
The inclusion $\supseteq$ follows by Remark \ref{basic}(c). Conversely, if $\star'\not\leq\star$, then there is a $A$-submodule $F$ of $K$ such that $F^{\star'}\nsubseteq F^\star$; hence, if $x\in F^{\star'}\setminus F^\star$, then $\star\notin V_{x^{-1}F}$ while $\star'\in V_{x^{-1}F}$. Therefore, $\inssemistar(A)\setminus V_{x^{-1}F}$ is a closed set containing $\star$ but not $\star'$, and $\star'\notin\chius(\{\star\})$.

Since $d$ is the smallest semistar operation, with respect to $\leq$, the last statement is now clear (see also Remark \ref{basic}(b)).
\end{proof}
\begin{prop}\label{quoziente-canonico}
We preserve the notation given at the beginning of the present section. Then, the following statements hold. 
\begin{enumerate}[\hspace{0.6cm}\rm (1)]
\item The topological space $\inssemistar(A)$ is $T_0$.
\item The canonical map $\Phi:\inssemistar(A)\longrightarrow \insfinss(A)$, $\star\mapsto\star_f$ is a topological retraction. 
\end{enumerate}
\end{prop}
\begin{proof} (1). By Proposition \ref{chiusura}, for any semistar operation $\star,\star'$ on $A$, we have $\ad(\{\star\})=\ad(\{\star'\})$ if and only if $\star=\star'$, i.e., $\inssemistar(A)$ satisfies {the} $T_0$ axiom. 

(2). The fact that $\Phi|_{\insfinss(A)}$ {is the identity} follows immediately by definition. Moreover, $\Phi$ is continuous since, for any $F\in \mf f(A)$, we have $\Phi^{-1}(U_F)=V_F$, and $\{U_F: F\in\mf f(A)\}$ is a subbase of $\insfinss(A)$ by Remark \ref{basic}(\ref{basic:ft}). Thus $\Phi$ is a topological retraction. 
\end{proof}

The next goal is to justify the fact that we have called the topology  on $\inssemistar(A)$ \emph{the Zariski topology.} First of all, recall that the set $\Over(A)$ of all the overrings of an integral domain $A$ can be endowed with the Zariski topology whose basic open sets are those of the form $B_F:=\{C\in \Over(A):F\subseteq C\}$, where $F$ ranges among the finite subsets of the quotient field of $A$ (see \cite{ol}), or, equivalently, among the finitely generated fractional ideals of $A$. As we saw in the previous section, we can associate to each $D\in \Over(A)$ the semistar operation (of finite type) $\wedge_{\{D\}}$ such that $F\mapsto FD$, for any $F\in \overline{\mf F}(A)$. Thus we can define {a} natural map $\phi:\Over(A)\longrightarrow \insfinss(A)$, $D\mapsto \wedge_{\{D\}}$, and, since obviously $A^{\wedge_{\{D\}}}=D$, for any $D\in \Over(A)$, we infer immediately that $\phi$ is injective. In the following Proposition we will show more. 
\begin{prop}\label{embedding}
We preserve the notation given at the beginning of this section and endow $\Over(A)$ and $\insfinss(A)$ with their Zariski topologies. Then the natural map $\phi:\Over(A)\longrightarrow \insfinss(A)$, $D\mapsto
\wedge_{\{D\}}$, is a topological embedding. 
\end{prop}
\begin{proof}
First, we show that $\phi$ is continuous, and it is enough to show that, for any finitely generated fractional ideal $F$ of $A$, the set $\phi^{-1}(U_F)$ is open in $\Over(A)$. Since 
$$
D\in\phi^{-1}(U_F)\iff \wedge_{\{D\}}\in U_F\iff 1\in F^{\wedge_{\{D\}}}\iff 1\in FD,
$$
we have $\phi^{-1}(U_F)=\{D\in \Over(A):1\in FD\}$. Fix a ring $D\in \phi^{-1}(U_F)$. Then there are $d_1,\ldots,d_n\in D$, $f_1,\ldots,f_n\in F$ such that $1=d_1f_1+\cdots+d_nf_n$. Hence $1\in FC$, for each $C\in B_{\{d_1,\ldots,d_n\}}$,  and thus $C\in\phi^{-1}(U_F)$. Therefore, $B_{\{d_1,\ldots,d_n\}}$ is an open neighborhood of $D$ contained in $\phi^{-1}(U_F)$, which is thus open.

Finally we show that the image via $\phi$ of an open set $V$ of $\Over(A)$ is open in $\phi(\Over(A))$ (endowed with the subspace topology). Without loss of generality, we can assume that $V=B_F$, for some finite subset $F:=\{f_1,\z,f_n\}$ of $K-\{0\}$. First, consider the open set $U:=\bigcap_{i=1}^nU_{f_i^{-1}}$ of $\insfinss(A)$. If $\star\in\phi(B_F)$, then $\star=\wedge_{\{C\}}$ for some $C\supseteq F$; hence $f_i\in C$ for every $i$ and $1\in f_i^{-1}C$, i.e., $1\in (f_i^{-1})^{\wedge_{\{C\}}}$. Thus  we have $\phi(B_F)\subseteq U\cap\phi(\Over(A))$.

Conversely, if $\star\in U\cap\phi(\Over(A))$, then $\star=\wedge_{\{C\}}$ and $1\in(f_i^{-1})^{\wedge_{\{C\}}}$ for every $i$; it follows that $f_i\in C$ for every $i$, and thus $C\in B_F$. The equality $\phi(B_F)=U\cap \phi(\Over(A))$ shows that $\phi(B_F)$ is open in $\phi(\Over(A))$. The proof is now complete. 
\end{proof}
The map $\phi$ defined above is only rarely surjective (or, equivalently, an homeomorphism). 
 We shall need the following fact.
\begin{oss}\label{t=d}
An integrally closed domain $R$ is a Pr\"ufer domain if and only if  $t=d$. Indeed, the latter condition is equivalent to the divisoriality of every nonzero finitely generated ideal. This happens if and only if $R$ is a Pr\"ufer domain, by \cite[proof of Theorem 5.1]{he}.
\end{oss}
By \cite[Theorem 2.3]{mim}, the map $\phi$ defined above is surjective if and only if, for every $D\in\Over(A)$, the unique \emph{star} operation of finite type is the identity. Every Pr\"ufer domain has this property; conversely, if $\phi$ is surjective, consider the integral closure $\overline{A}$ of $A$. Then, $t=d$ on $\overline{A}$, and thus $\overline{A}$ is a Pr\"ufer domain, by Remark \ref{t=d}.

We now investigate the algebraic interpretation of compactness for the subspaces of $\insfinss(A)$. 
\begin{prop}\label{compact-fintype}
We preserve the notation given at the beginning of the present section, and let $\Delta$ be a compact subspace of $\insfinss(A)$. Then, the semistar operation $\bigwedge(\Delta)$ is of finite type. 
\end{prop}
\begin{proof}
Set $\Delta:=\{\star_i:i\in I\}$,  $\star:=\bigwedge(\Delta)$, fix a $A$-submodule $F$ of $K$ and let $x\in F^\star$. Since $F^\star=\bigcap_{i\in I}F^{\star_i}$, and each $\star_i$ is of finite type, there are finitely generated ideals $G_i\subseteq F$ such that $x\in G_i^{\star_i}$; thus, for any $i$, $1\in x^{-1}G_i^{\star_i}=(x^{-1}G_i)^{\star_i}$ and $\star_i\in U_{x^{-1}G_i}=:\Omega_i$. Therefore, $\{\Omega_i : i\in I\}$ is an open cover of $\Delta$, and by compactness it admits  a finite subcover $\{\Omega_{i_1},\ldots,\Omega_{i_n}\}$. Set $G:=G_{i_1}+\cdots+G_{i_n}\subseteq F$; we claim that $x\in G^\star$, and this implies that $\star$ is of finite type.

For every $i\in I$, there is at least a $\Omega_{j_i}$ such that $\star_i\in \Omega_{j_i}$; hence $\star_i\in U_{x^{-1}G_{j_i}}$ and $1\in(x^{-1}G_{j_i})^{\star_i}$, i.e., $x\in G_{j_i}^{\star_i}\subseteq G^{\star_i}$. Therefore, $x\in\bigcap_{i\in I}G^{\star_i}=G^\star$.
\end{proof}
\begin{cor}\label{comp-overrings}
We preserve the notation given at the beginning of the section and let $Y$ be a compact subspace of $\Over(A)$. Then, the semistar operation $\wedge_Y$ is of finite type. 
\end{cor}
\begin{proof}
Apply Propositions \ref{embedding} and \ref{compact-fintype}.
\end{proof}

Recall that a subspace $Y$ of $\Over(A)$ is \emph{locally finite} if any nonzero element of $A$ is non-invertible only in finitely many rings of $Y$. The following result generalizes \cite[Theorem 2(4)]{an2}.
\begin{prop}\label{prop:locfin}
We preserve the notation given at the beginning of this section. Let $\{B_i:i\in I\}$ be a locally finite family of overrings of $A$ and, for any $i\in I$, let $\star_i$ be a semistar operation of finite type on $B_i$. Then the map $\star:F\mapsto\bigcap_{i\in I}(FB_i)^{\star_i}$ is a semistar operation of finite type on $A$.
\end{prop}
\begin{proof}
Let $\star_i^\sharp$ be the map $\star_i^\sharp:F\mapsto(FB_i)^{\star_i}$. Borrowing Proposition \ref{funtoriale1}(\ref{funtoriale1:fintype}), we note that $\star_i^\sharp$ is a semistar operation on $A$ of finite type, since $\star_i$ is of finite type on $B_i$. Moreover, $\star=\bigwedge(\Delta)$, where $\Delta:=\{\star_i^\sharp: i\in I\}$, and by Proposition \ref{compact-fintype} it suffices to show that $\Delta$ is compact.

Let $\mathcal U$ be an open cover of $\Delta$. By Alexander's subbasis Theorem, we can assume that each set in $\mathcal U$ is a subbasic open set of the Zariski topology. Choose an ideal $F\in \mf f(A)$ such that $U_F\in \mathcal U$ and let $x_0\in F-\{0\}$. By local finiteness, there is a finite subset $I'\subseteq I$ such that $x_0,x_0^{-1}\in B_i$, for any $i\in I-I'$. Thus we have $1=x_0x_0^{-1}\in FB_i\subseteq (FB_i)^{\star_i}=:F^{\star_i^\sharp}$, for any $i\in I-I'$. For every $i\in I$, there is a $F_i$ such that $\star_i^\sharp\in U_{F_i}$; hence, $\{U_{F_i}:i\in I'\}\cup\{U_F\}$ is a finite subcover of $\mathcal U$, and thus $\Delta$ is compact.
\end{proof}
\begin{cor}
Let $Y$ be a locally finite subset of $\Over(A)$. Then, $\wedge_Y$ is of finite type. 
\end{cor}
\begin{proof}
Apply the previous Proposition by taking the identity semistar operation on each ring in $Y$.  
\end{proof}
We now show that the order structure of the intersection of a nonempty family of subbasic open sets of the Zariski topology is particularly simple. 
\begin{prop}\label{aperti-compatti}
We preserve the notation of the beginning of this section and let $\{V_{F_i}:i\in I\}$ be a nonempty family of subbasic open sets of the Zariski topology of $\inssemistar(A)$. The following statements hold. 
\begin{enumerate}[\hspace{0.6cm}\rm (1)]
\item $\bigcap\{V_{F_i}:i\in I\}$ is a complete lattice (as a subset of the partially ordered set $(\inssemistar(A),\leq)$).
\item $\bigcap\{V_{F_i}:i\in I\}$ is a compact subspace of $\inssemistar(A)$. In particular, $V_F$ is compact for every $F\in \inssubmod(A)$. 
\end{enumerate}
\end{prop}
\begin{proof}
Set $V:=\bigcap_{i\in I}V_{F_i}$ and  let $\Delta$ be a nonempty subset of $V$. By Example \ref{sect:background-semistar}(\ref{bg:infsup}), $\sharp:=\bigvee(\Delta)$ and $\flat:=\bigwedge(\Delta)$ are, respectively, the supremum and the infimum of $\Delta$ in $\inssemistar(A)$. Thus it suffices to show that $\sharp,\flat\in V$. Clearly, $\sharp\in V$, because $\sharp\geq \star$, for any $\star\in\Delta$, and thus $\sharp$ belongs to each open set containing some $\star$, by Remark \ref{basic}(c). Furthermore, for any $i\in I$ we have
$$
1\in \bigcap_{\star\in V_{F_i}}F_i^\star\subseteq \bigcap_{\star\in V}F_i^\star=:F_i^\flat,
$$
and thus $\flat\in V$. Statement (1) is completely proved. 

Now let $\mathcal U$ be an open cover of $V$. By (1), $\flat\in V$, and thus there is an open set $U_0\in \mathcal U$ such that $\flat\in U_0$. Finally, by Remark \ref{basic}(c), $U_0$ must contain the whole $V$. Statement (2) is now clear. 
\end{proof}
The next goal is to show that $\insfinss(A)$, endowed with the Zariski topology, is a spectral space. To do this, we will use the following lemma, whose proof, \emph{mutatis mutandis}, is based on the argument given in \cite[p. 1628]{an} in the star operation setting. If $\sigma_1,\z,\sigma_n$ are semistar operations on $A$ we denote by $\sigma_1\circ\z\circ\sigma_n$ the usual composition of $\sigma_1,\z,\sigma_n$ as functions. 
\begin{lemma}\label{sup}
Let $Y$ be a nonempty collection of semistar operations of finite type on an integral domain $A$. Then $\bigvee (Y)$ is of finite type and, for any $F\in \overline{\mf F}(A)$, we have
$$
F^{\bigvee(Y)}=\bigcup\{F^{\sigma_1\circ \z\circ\sigma_n}:\sigma_1,\z,\sigma_n\in Y,n\in\mathbb{N}\}.
$$
\end{lemma}
\begin{teor}\label{spectral}
Preserve the notation given at the beginning of the present section. The set $\insfinss(A)$, endowed with the Zariski topology, is a spectral space. 
\end{teor}
\begin{proof}
Let $\ms U$ be an ultrafilter on $X:=\insfinss(A)$ and let $\mathcal S:=\{U_F:F\in \mf f(A)\}$ be the canonical subbasis of the Zariski topology. In view of \cite[Corollary 3.3]{Fi}, it suffices to show that the set
$$
X_{\mathcal S}(\ms U):=\{\star\in X: [\forall U_F\in \mathcal S, \star\in U_F\iff U_F\in \ms U]\}
$$
is nonempty. By Propositions \ref{compact-fintype} and \ref{aperti-compatti}, any semistar operation of the form $\bigwedge(U_F)$ (where $F\in \mf f(A)$) is of finite type. Thus the semistar operation $\star:=\bigvee(\{\bigwedge(U_F):U_F\in \ms U\})$ is of finite type. We claim that  $\star\in X_{\mathcal S}(\ms U)$. Fix a finitely generated fractional ideal $F$ of $A$. It suffices to show that $\star\in U_F$ if and only if $U_F\in \ms U$. First, assume $\star\in U_F$, i.e., $1\in F^\star$. By Lemma \ref{sup}, there exist finitely generated fractional ideals $F_1,\z,F_n$ of $A$ (not necessarily distinct) such that $1\in F^{\bigwedge(U_{F_1})\circ \z\circ \bigwedge(U_{F_n})}$ and $U_{F_i}\in \ms U$, for any $i=1,\z,n$. Take a semistar operation $\sigma\in \bigcap_{i=1}^nU_{F_i}$. By definition, $\sigma\geq \bigwedge(U_{F_i})$, for $i=1,\z,n$, and thus 
$$
1\in F^{\bigwedge(U_{F_1})\circ \z\circ \bigwedge(U_{F_n})}\subseteq F^{\sigma\circ\z\circ\sigma}=F^\sigma,
$$
i.e, $\sigma\in U_F$. This shows that $\bigcap_{i=1}^nU_{F_i}\subseteq U_F$ and thus, by definition of ultrafilter, $U_F\in\ms U$, since $U_{F_1},\z,U_{F_n}\in \ms U$. Conversely, assume that $U_F\in\ms U$. This implies that $\bigwedge(U_F)\leq \star$. By definition, $1\in F^\sigma$, for each $\sigma\in U_F$, and thus 
$$
1\in \bigcap_{\sigma\in U_F}F^\sigma=:F^{\bigwedge(U_F)}\subseteq F^\star.
$$
The conclusion is now clear. 
\end{proof}
\begin{cor}
Preserve the notation given at the beginning of the present section. The set $\insss(A)$ of all the (semi)star operations of finite type, endowed with the subspace Zariski topology, is a spectral space. 
\end{cor}
\begin{proof}
For every $F\in \mf f(A)$, let $U'_F:=\insss(A)\cap U_F$. Clearly, a subbasis of the subspace topology of $\insss(A)$ is $\{U'_F:F\in \mf f(A)   \}$. Moreover, if $U'_F\neq \emptyset$, then its infimum is equal to the infimum of $U_F$, and the supremum of a family of (semi)star operations is still a (semi)star operation. Hence, in the proof of Theorem \ref{spectral}, the map $\star:=\bigvee(\{\bigwedge(U'_F):U'_F\in \ms U\})$ is a (semi)star operation, and thus we can apply the same proof to get the statement.
\end{proof}
\section{Functorial properties}\label{sect:funct}
Let $A\subseteq B$ be an extension of integral domains, and  let $K$ be the quotient field of $A$. For any semistar operation $\star\in \inssemistar(B)$ we can define a semistar operation $\sigma(\star)\in \inssemistar(A)$ by setting
$$
F^{\sigma(\star)}:=(FB)^\star\cap K,
$$
for every nonzero $A$-submodule $F$ of $K$. Thus the inclusion induces a natural map $\sigma:\inssemistar(B)\longrightarrow \inssemistar(A)$, $\star\mapsto\sigma(\star)$. 
\begin{prop}\label{funtoriale1}
Preserve the notation given at the beginning of the present section. The following statements hold. 
\begin{enumerate}[\hspace{0.6cm}\rm (1)]
\item The map $\sigma$ is continuous.
\item\label{funtoriale1:fintype} If $\star\in \inssemistar(B)$ is of finite type, then so is $\sigma(\star)$; i.e., $\sigma$ restricts to a map $\insfinss(B)\longrightarrow\insfinss(A)$.
\end{enumerate}

\end{prop}
\begin{proof}(1)
Let $F\in\inssubmod(A)$. Then we have:
$$
\sigma^{-1}(V_F^{(A)})=\{\star\in\inssemistar(B): \sigma(\star)\in V_F\}=\{\star\in\inssemistar(B): 1\in F^{\sigma(\star)}\}=
$$
$$
=\{\star\in\inssemistar(B): 1\in (FB)^\star\}=V_{FB}^{(B)}.
$$
Since $FB\in\inssubmod(B)$, $\sigma$ is continuous.

(2) Let $I\in \overline{ \mf F}(A)$ and $x\in I^{\sigma(\star)}$; then $x\in(IB)^\star$, and thus there are $y_1,\ldots,y_n\in IB$ such that $x\in(y_1B+\cdots+y_nB)^\star$. For every $y_i$, there is a finitely generated $A$-module $F_i\subseteq I$ such that $y_i\in F_iB$; let $F:=F_1+\cdots+F_n$. Then $F\subseteq I$ is finitely generated (as an $A$-module), and $y_1B+\cdots+y_nB\subseteq FB$; therefore, $x\in(FB)^\star$ and $x\in F^{\sigma(\star)}$. Thus $\sigma(\star)$ is of finite type.
\end{proof}
We note in the following result that the map $\sigma$ defined above exhibits better properties when $A$ and $B$ have the same quotient field.
\begin{prop}\label{prop:funct-over}
We preserve the notation of the beginning of the present section, and suppose in addition that $B$ is an overring of $A$. If $\star\in \inssemistar(B)$,  the following statements hold. 
\begin{enumerate}[\hspace{0.6cm}\rm (1)]
\item $\sigma(\star)|_{\inssubmod(B)}=\star$;
\item $\sigma$ is injective;
\item $\sigma(\star)$ is of finite type if and only if so is $\star$.
\end{enumerate}
\end{prop}
\begin{proof}
The first point is straighforward, since if $I$ is a $B$-module, then $I^{\sigma(\star)}=(IB)^\star=I^\star$; the second follows immediately from the first one.

For the third one, suppose $\sigma(\star)$ is of finite type and let $x\in I^\star$. Then $x\in I^{\sigma(\star)}$, and thus there is a finitely generated $A$-module $F\subseteq I$ such that $x\in F^{\sigma(\star)}$. Hence, $x\in(FB)^\star$, and $\star$ is of finite type. By Proposition \ref{funtoriale1}(2), the proof is complete. 
\end{proof}
\begin{oss}
If the quotient fields of $A$ and $B$ are different, then $\sigma$ could fail to be injective: the simplest example is obtained when $B$ is not a field but contains the quotient field $K$ of $A$, since for every $\star\in\inssemistar(B)$ and every $I\in\inssubmod(A)$, $I^{\sigma(\star)}=(IB)^\star\cap K=K^\star\cap K=K$. The injectivity is not preserved even if we suppose $B\cap K=A$: for example, let $A$ be a rank one discrete valuation ring, $L$ a field containing $K$ and define $B$ as the integral closure of $A$ in $L$. If $B$ is not local (e.g., if $A=\mathbb{Z}_{(p)}$ and $L$ is an algebraic extension of $\mathbb{Q}$ where $p$ splits), then $\sigma$ is not injective, since $|\inssemistar(A)|=2$ (\cite[Proposition 4.2]{pi}) while $B$ admits at least 3 semistar operation of finite type: the identity, $\wedge_{\{L\}}$ and $\wedge_{\{B_P\}}$, where $P$ is a maximal ideal of $B$.

In the same way, $\sigma(\star)$ could be of finite type even if $\star$ is not and $B\cap K=A$: for example, let $Z$ be an indeterminate over $\mathbb{C}$, set $A=\mathbb{C}[Z]$ and let $B$ be the ring of all entire functions. Then the map $\star$ defined by $$F\mapsto F^\star:=\bigcap_{\alpha\in\mathbb{C}}FB_{(Z-\alpha)}\qquad \mbox{ for any }F\in\inssubmod(B)$$ is a (semi)star operation on $B$ which is not of finite type. Indeed, since $B$ is a B\'ezout domain \cite{entire}, all finitely generated ideals are quasi-$\star$-ideals but, if $\f b\subsetneq B$ is \emph{free ideal} (i.e., the functions belonging to $\f b$ have no common zeros), then clearly $\f b^\star=B$, while for any finite subset $\{f_1,\z,f_n\}$ of $\f b$, we have $((f_1,\z,f_n)B)^\star= (f_1,\z,f_n)B\subseteq \f b\subsetneq B$. This shows that $\star$ is not of finite type.  Since $B\cap\mathbb{C}(Z)=A$, $\sigma(\star)$ is a (semi)star operation on $A$, and it is not hard to see that it is the identity, a!
 nd thus of finite type.
\end{oss}
\begin{prop}
We preserve the notation of the beginning of the present section and suppose that $B$ is an overring of $A$. Then 
$$
\sigma(\inssemistar(B))=\{\star\in\inssemistar(A): \star\geq\wedge_{\{B\}}\}=\{\star\in\inssemistar(A): B\subseteq A^\star\}.
$$
\end{prop}
\begin{proof}
If $\star\in\sigma(\inssemistar(B))$, then $\star=\sigma(\star')$, for some $\star'\in \inssemistar(B)$, and, for every $I\in\inssubmod(A)$,
$$
I^{\sigma(\star')}=(IB)^\star\supseteq IB=I^{\wedge_{\{B\}}}
$$
and thus $\star\geq\wedge_{\{B\}}$.

If $\star\geq\wedge_{\{B\}}$, then $A^\star\supseteq A^{\wedge_{\{B\}}}=AB=B$.

 If $B\subseteq A^\star$ and  $I\in \inssubmod(A)$, then $I^\star$ is a $B$-module, since for every $b\in B$ we have $bI^\star=(bI)^\star\subseteq(A^\star I)^\star=(AI)^\star\subseteq I^\star$. Hence, $\star|_{\inssubmod(B)}\in\inssemistar(B)$ and it is easy to see that $\star=\sigma(\star)$.
\end{proof}

\begin{prop}
We preserve the notation of the beginning of the present section and assume that $B$ is an overring of $A$. Then the map $\sigma:\inssemistar(B)\longrightarrow\inssemistar(A)$ is a topological embedding. 
\end{prop}
\begin{proof}
Let $F\in\inssubmod(B)$. By Propositions \ref{funtoriale1}(1) and \ref{prop:funct-over}(2), it is enough to show that $\sigma(V_F^{(B)})$ is open in $\sigma(\inssemistar(B))$.

Since $B$ is an overring of $A$, then $F$ is an $A$-module and thus is defined the open set $V_F^{(A)}$. But, since $F^\star=F^{\sigma(\star)}$ for every $\star\in\inssemistar(B)$ (Proposition \ref{prop:funct-over}), we have $\sigma(V_F^{(B)})=V_F^{(A)}\cap\sigma(\inssemistar(B))$, which is an open set in $\sigma(\inssemistar(B))$. The proof is now complete.
\end{proof}
\section{Spaces of local rings}
Let $A\subseteq B$ be a ring extension. Let ${\rm L}(B|A)$ denote the set (possibly empty) of the local subrings of $B$ containing $A$. We can define on ${\rm L}(B|A)$ a topology just by taking as a basis of open sets the collection of the sets of the form ${\rm L}(B|A[F])$, where $F$ runs in the collection of all the finite subsets of $B$. We will cal this topology \emph{the Zariski topology on ${\rm L}(B|A)$}. When $A$ is an integral domain and $B$ is the quotient field of $A$, then ${\rm L}(A):={\rm L}(B|A)$ is simply the space of all the local overrings of $A$. By \cite[Lemma 2.4]{do-fe-fo}, there is a canonical topological embedding $$\iota: \spec(A)\longrightarrow {\rm L}(A)\qquad \f p\mapsto A_{\f p}.$$

Moreover, ${\rm L}(A)$ is obviously a subset of $\rm{Over}(A)$, and the inclusion ${\rm L}(A)\longrightarrow\Over(A)$ is easily seen to be a topological embedding, when the two sets are endowed with the respective Zariski topologies.
\begin{lemma}\label{lemma1}
Let $A\subseteq B$ be a ring extension and consider the canonical map $\lambda:{\rm L}(B|A)\longrightarrow \spec(A)$ sending a local ring $C\in {\rm L}(B|A)$ with maximal ideal $\f m_C$ into the prime ideal $\f m_C\cap A$ of $A$.
\begin{enumerate}[\hspace{0.6cm}\rm (1)]
\item $\lambda$ is continuous.
\item If $A$ is an integral domain and $B$ is the quotient field of $A$, then $\lambda$ is a topological retraction.
\end{enumerate} 
\end{lemma}
\begin{proof}
It suffices to note that 
$
\lambda^{-1}(D(f))={\rm L}(B|A[f^{-1}]),
$
for any element $f\in A$.
\end{proof}
\begin{prop}\label{lover-comp}
Let $A$ be an integral domain, $Y$ be a nonempty subspace of ${\rm L}(A)$ and assume that $\wedge_Y$ is a semistar operation of finite type. If $\lambda:{\rm L}(A)\longrightarrow\spec(A)$ is the canonical continuous map, then $\lambda(Y)$ is compact. 
\end{prop}
\begin{proof}
Let $\{D(f_i):i\in I\}$ ($f_i\in A$) be a collection of basic open sets of $\spec(A)$ such that $\bigcup\{D(f_i):i\in I\}\supseteq \lambda(Y)$. Thus for any $B\in Y$ there is an element $i_B\in I$ such that $f_{i_B}\notin \f m_B\cap A$. Since $B$ is local and $f_{i_M}\in A\subseteq B$, we have $f_{i_B}^{-1}\in B$. Thus, if $\f a$ is the ideal of $A$ generated by the set $\{f_i:i\in I\}$, we infer immediately that $1\in \f aB$, for any $B\in Y$, i.e., $1\in \f a^{\wedge_Y}$. Since $\wedge_Y$ is of finite type, there is a finitely generated ideal $\f b$ of $A$ contained in $\f a$ such that $1\in \f b^{\wedge_Y}$. Of course, we can assume, without loss of generality, that a finite set of generators of $\f b$ is $\{f_j:j\in J\}$, where $J$ is a suitable finite subset of $I$. It suffices to show that $\bigcup\{D(f_j):j\in J\}\supseteq\lambda(Y)$. If, for some $B\in Y$, we had  $(f_j:j\in J)\subseteq \f m_B\cap A$, it would follow that $\f bB\subseteq \f m_B$, against the fact tha!
 t $1\in \f b^{\wedge_Y}$. The proof is now complete.
\end{proof}
The following corollary is an immediate consequence of Propositions \ref{compact-fintype} and \ref{lover-comp}.
\begin{cor}
Preserve the notation of Proposition \ref{lover-comp} and let $Y$ be a subspace of ${\rm L}(A)$ such that $\lambda|_Y$ is a topological embedding. Then $\wedge_Y$ is a semistar operation of finite type if and only if $Y$ is compact. 
\end{cor}
The following Corollary is an improvement of \cite[Corollary 4.6]{fohu}.
\begin{cor}\label{fin-spectral-car}
Let $A$ be an integral domain and let $\Delta\subseteq \spec(A)$. Then $s_\Delta$ is a semistar operation of finite type if and only if $\Delta$ is compact. 
\end{cor}
\begin{proof}
Apply the previous Corollary, keeping in mind that the restriction of $\lambda:{\rm L}(A)\longrightarrow\spec(A)$ to the set $\iota(\spec(A))$ is clearly an homeomorphism (cf. \cite[Lemma 2.4]{do-fe-fo}). 
\end{proof}

\begin{prop}\label{valuations}
Let $A$ be an integral domain and $Y$ be a nonempty collection of valuation overrings of $A$. Then, $\wedge_Y$ is of finite type if and only if $Y$ is a compact subspace of ${\rm L}(A)$.
\end{prop}
\begin{proof}
The sufficient condition is proved in \cite[Theorem 4.13]{fifolo2}. It follows also from Corollary \ref{comp-overrings}, since {\rm L}(A) is a subspace of $\Over(A)$.

Conversely, assume that $\wedge_Y$ is of finite type and let $\mathcal U$ be an open cover of $Y$. Clearly, a subbasis of open sets of $Y$, as a subspace of ${\rm L}(A)$, consists of the sets of the form $B_f:=\{V\in Y:f\in V \}$, for any element $f$ of the quotient field $K$ of $A$. Thus, by Alexander's subbasis Theorem, we can assume that $\mathcal U=\{B_{f_i}:i\in I\}$ and that every $f_i$ is nonzero. If $F$ is the $A$-submodule of $K$ generated by the set $\{f_i^{-1}:i\in I\}$, then, by definition, $1\in F^{\wedge_Y}$ and, since $\wedge_Y$ is of finite type, there is a finitely generated $A$-submodule $G$ of $F$ such that $1\in G^{\wedge_Y}$. Of course, we can assume that $G=(\{f_j^{-1}:j\in J \})$, where $J$ is some finite subset of $I$.  We claim that $\{B_{f_j}:j\in J\}$ is a (finite) subcover of $\mathcal U$. If not, there is a valuation domain $V\in Y$ such that $f_j\notin V$, for any $j\in J$, and hence $f_j^{-1}$ is an element of the maximal ideal $\f m$ of $V$, f!
 or
  any $j\in J$. Since $1\in G^{\wedge_Y}$, we infer, in particular, that $1\in GV\subseteq \f m$, a contradiction. The proof is now complete. 
\end{proof}
\begin{oss}
If $Y$ is the space of all the valuation overrings of an integral domain, then $\wedge_Y=b$ is the integral closure. It is well known (see e.g. \cite[Appendix 4]{zar-sam} or \cite[Section 6.8]{sw-hun}) that an element $x\in K$ is in $I^b$ if and only if there is an integer $n$ and there are elements $a_1,\ldots,a_n\in K$ such that $a_i\in I^i$ and $x^n+a_1x^{n-1}+a_2x^{n-2}+\cdots+a_n=0$. Using this equivalence, it is easy to see that $b$ is of finite type, and thus Proposition \ref{valuations} implies that $Y$ is compact.
\end{oss}

In view of Corollaries \ref{comp-overrings} and \ref{fin-spectral-car} and Proposition \ref{valuations}, if $Y$ is either a collection of localization of an integral domain $A$ or a collection of valuation overrings of $A$, then compactness of $Y$ is equivalent to the requirement that $\wedge_Y$ is of finite type. This gives evidence to the following 

{\bf Conjecture.} Let $Y$ be \emph{any} subspace of ${\rm L}(A)$. If $\wedge_Y$ is of finite type, then $Y$ is compact. 

\section{Spectral semistar operations}\label{sect:spectral}
Let $A$ be an integral domain and $K$ be the quotient field of $A$. As we saw in Section 1, a semistar operation $\star$ on $A$ is spectral if there is a nonempty set $Y$ of prime ideals of $A$ such that $\star=s_Y$, i.e.
$$
F^\star=F^{s_Y}:=\bigcap_{\f p\in Y}FA_{\f p} \quad \mbox{ for any } F\in \inssubmod(A).
$$
Following \cite{pi-ta}, two semistar operations $\star_1,\star_2$ on $A$ are called \emph{weakly equivalent} if $\widetilde{\star_1}=\widetilde{\star_2}$. 
\begin{prop}\label{weakly-eq}
Let $A$ be an integral domain and let $Y,Z$ be nonempty subsets of $\spec(A)$. Then, the following conditions are equivalent.
\begin{enumerate}[\hspace{0.6cm}\rm (i)]
\item $s_Y$ and $s_Z$ are weakly equivalent.
\item $\ad^{\rm i}(Y)=\ad^{\rm i}(Z)$, that is, $Y$ and $Z$ have the same closure, with respect to the inverse topology. 
\item If $\f a$ is a finitely generated ideal of $A$ and $D(\f a):=\spec(A)-V(\f a)$, then $Y\subseteq D(\f a)$ if and only if $Z\subseteq D(\f a)$. 
\end{enumerate}
\end{prop}
\begin{proof}
As a preliminary, note that if $\f a:=(f_1,\z,f_n)A$ is a finitely generated ideal of $A$, then $D(\f a)=\bigcup_{i=1}^nD(f_i)$ is an open and compact subspace of $\spec(A)$ and thus it is, by definition, a basic closed set, with respect to the inverse topology. Conversely, if $D(\f a)$ is compact and $\f a=(f_\lambda)_{\lambda\in\Lambda}$, then $D(\f a)=\bigcup_{\lambda\in\Lambda} D(f_\lambda)=\bigcup_{i=1}^n D(f_{\lambda_i})=D(\f b)$, where $\f b=(f_{\lambda_1},\ldots,f_{\lambda_n})$ is finitely generated.

Suppose that condition (ii) holds. Thus, for any finitely generated ideal $\f a$ of $A$, we have $Y\subseteq D(\f a)$ if and only if $\ad^{\rm i}(Z)=\ad^{\rm i}(Y)\subseteq D(\f a)=\ad^{\rm i}(D(\f a))$, if and only if $Z\subseteq D(\f a)$; hence, (iii) holds.

Conversely, assume that condition (iii) holds, and let $\Omega$ be an open and compact subspace of $\spec(A)$, i.e., $\Omega=D(\f a)$, for some finitely generated ideal $\f a$ of $A$. Then $\Omega\supseteq Y$ if and only if $\Omega\supseteq Z$, i.e., by definition, $\ad^{\rm i}(Y)=\ad^{\rm i}(Z)$, and (ii) holds.

To conclude, we show that (i) and (iii) are equivalent. By \cite[Proposition 2.4(iii)]{pi-ta}, $s_Y$ and $s_Z$ are weakly equivalent if and only if for any finitely generated ideal $\f a$ of $A$, we have $\f a^{s_Y}=A^{s_Y}\iff \f a^{s_Z}=A^{s_Z}$. By basic properties of localizations, the last statement is clearly equivalent to condition (iii). The proof is now complete. 
\end{proof}
\begin{cor}\label{stable-closure}
Let $A$ be an integral domain and $Y$ be a subset of $\spec(A)$. Then $\widetilde{s_Y}=s_{\ad^{\rm i}(Y)}$.
\end{cor}
\begin{proof}
By Section \ref{inv-rem}, the set $\ad^{\rm i}(Y)$ is compact, with respect to the Zariski topology. Hence the stable semistar operation $s_{\ad^{\rm i}(Y)}$ is also of finite type, by Proposition \ref{compact-fintype}, and thus $s_{\ad^{\rm i}(Y)}=\widetilde{s_{\ad^{\rm i}(Y)}}$. The conclusion follows now immediately from Proposition \ref{weakly-eq}. 
\end{proof}

\begin{oss} Note that the finite type closure of a spectral semistar operation may be not spectral (see \cite[pag. 2466]{an-co}). For example, let $A$ be an essential domain that is not a P$v$MD (see \cite{he-oh}), and let $Y$ be a subset of $\spec(A)$ such that $A_{\f p}$ is a valuation domain, for any $\f p\in Y$ and $A=\bigcap\{A_{\f p}:\f p\in Y\}$. By \cite[Proposition 44.13]{gi} we have $(s_Y)_f=t$, but $t$ is not spectral since the fact that $A$ is not a P$v$MD implies that there are finitely generated fractional ideals $F,G$ of $A$ such that $(F\cap G)^t\neq F^t\cap G^t$ \cite[Theorem 6]{an2}.
\end{oss}
An integral domain is called \emph{a DW domain} if the semistar operations $d$ and $w:=\widetilde t$ are the same. For example, all Pr\"ufer domains are DW domains, as are one-dimensional domains.
\begin{cor}
For an integral domain $A$, the following conditions are equivalent. 
\begin{enumerate}[\hspace{0.6cm}\rm (i)]
\item $A$ is a DW domain.
\item Every $Y\subseteq\spec(A)$ such that $\bigcap_{P\in Y}A_P=A$ is dense in $\spec(A)$, with respect to the inverse topology.
\item $\qmax^t(A)$ is dense in $\spec(A)$, with respect to the inverse topology. 
\end{enumerate}
\end{cor}
\begin{proof}
(i$\Longrightarrow$ ii) If $A$ is a DW domain and $\bigcap_{P\in Y}A_P=A$, then $s_Y$ is a star operation and therefore $\tilde{s}_Y\leq\tilde{v}=w=d=s_{\spec(A)}$; hence $Y$ is dense in $\spec(A)$.

(ii$\Longrightarrow$ iii) It is enough to note that $\bigcap_{P\in\qmax^t(A)}A_P=A$.

(iii$\Longrightarrow$ i) If $\qmax^t(A)$ is dense, then $w=s_{\qmax^t(A)}=s_{\spec(A)}=d$.
\end{proof}

Following \cite{folo}, we say that a semistar operation $\star$ on an integral domain $A$ is  \emph{quasi-spectral} or \emph{semifinite} if every proper quasi-$\star$-ideal is contained in some quasi-$\star$-prime ideal. Of course, this condition is equivalent to the requirement that, for any ideal $\f a$ of $A$, we have $\qspec^\star(A)\subseteq D(\f a):=\spec(A)-V(\f a)$ if and only if $\f a^\star\cap A=A$. 

The following example will show that the class of the semifinite semistar operations is very ample.
\begin{ex}\label{semifinite-example}
Let $A$ be an integral domain. 
\begin{enumerate}[(a)]
\item If $\star$ is a finite type semistar operation, then every proper quasi-$\star$ ideal is contained in some quasi-$\star$-maximal ideal, which is prime (see Section \ref{sect:background-semistar}). Thus every semistar operation of finite type is semifinite.
\item\label{semifinite-example:inf} If $\mathcal S:=\{\star_i:i\in I   \}$ is a nonempty family of semifinite semistar operations on $A$, then $\star:=\bigwedge(\mathcal S)$ is semifinite. Indeed, let $\f a$ be a proper quasi-$\star$-ideal of $A$. Then, by definition, for some $i_0\in I$, we have $1\notin \f a^{\star_{i_0}}$, i.e., $\f a^{\star_{i_0}}\cap A$ is a proper quasi-$\star_{i_0}$-ideal of $A$. By assumption, there is a quasi-$\star_{i_0}$-prime ideal $\f p$ containing $\f a^{\star_{i_0}}\cap A$ and, since $\star\leq \star_{i_0}$, $\f p$ is a quasi-$\star$-prime ideal. Finally $\f a=\f a^\star\cap A\subseteq \f a^{\star_{i_0}}\cap A\subseteq \f p$. 
\item By part (a) and (b), every semistar operation of the form $\wedge_Y$, where $Y$ is a nonempty subspace of $\Over(A)$, is semifinite. 
\item Not all semistar operations are semifinite: for example, if $(V,M)$ is a valuation domain of dimension 1 which is not discrete, then $M^v=V$, and thus every nonzero principal ideal is a $v$-ideal which is not contained in any $v$-prime. In general, if $(V,M)$ is a valuation domain, the $v$-operation is semifinite if and only if $M^v=M$, if and only if the $v$-operation coincides with the identity.
\end{enumerate}
\end{ex}

\begin{oss}
As we saw in Example \ref{semifinite-example}(a), every semistar operation of finite type is semifinite. The converse is not true. To see this, by Corollary \ref{fin-spectral-car} and Example \ref{semifinite-example}(c), it suffices to  consider a spectral semistar operation of the form $s_\Delta$, where $\Delta$ is a non-compact collection of prime ideals. For a somewhat extreme example, consider the ring $A:=K[X_0,\ldots,X_n,\ldots]$ of the polynomials in infinitely many indeterminates over a field $K$, and let $\Delta$ be the set of finitely generated prime ideals of $A$. We claim that $s_\Delta$ has no quasi-$\star$-maximal ideals. Indeed, if $M\in\qmax^{s_\Delta}(A)$, then $M^{s_\Delta}\neq A^{s_{\Delta}}$, and thus $MR_P\neq R_P$ for some finitely generated prime ideal $P$. Since $P\in\qspec^{s_\Delta}(A)$, it follows that $M=P$, i.e., $M$ is finitely generated. Let $M=(f_1,\ldots,f_n)$, and choose an indeterminate $X_N$ which does not appear in any $f_i$. Then the ide!
 al $M'$ generated by $M$ and $X_N$ is prime and finitely generated, and thus it is in $\qspec^{s_\Delta}(A)$. However, $M'$ is also strictly greater than $M$, against the maximality of $M$.
\end{oss}
While Remark \ref{semifinite-example}(\ref{semifinite-example:inf}) implies that the infimum of a family of finite type semistar operations is semifinite, not every semifinite operation arises this way, as the following example shows.
\begin{ex}
 Let $A$ be a non-local Dedekind domain and let $K$ be the quotient field of $A$: then, $A$ admits proper overrings different from $K$, and every proper overring of $A$ is not a fractional ideal of $A$. In fact, if $B\in\Over(A)$ is a fractional ideal of $A$, then it is finitely generated as an $A$-module, and thus integral over $A$. On the other hand, $A$ is integrally closed, since it is Dedekind, and thus $A=B$. 

Define a semistar operation $\star$ on $A$ by setting
\begin{equation*}
F^\star:=\begin{cases}
F & \text{if~} F\in\mf F(A)\\
K & \text{if~} F\in\inssubmod(A)-\mf F(A)\\
\end{cases}
\end{equation*}
Clearly, $\star$ is semifinite (every prime ideal is $\star$-closed) but not of finite type since, if $B$ is a proper overring of $A$ different from $K$, we have $B^\star=K$ and, on the other hand,
$$
\bigcup\{F^\star:F\in\mf f(A),F\subseteq B  \}=\bigcup\{F:F\in\mf f(A),F\subseteq B  \}=B\subsetneq K.
$$
Let $\sharp$ be a semistar operation of finite type such that $\sharp\geq\star$. We claim that $\sharp=\wedge_{\{K\}}$. Since $\sharp\geq \star$,  every $\sharp$-closed nonzero $A$-submodule  of $K$ is $\star$-closed. Since it is well known that $A^\ast$ is an overring of $A$ for every semistar operation $\ast$, $A^\sharp$ is a $\star$-closed overring of $A$, and thus either $A^\sharp=A$ or $A^\sharp=K$. In the first case, $\sharp|_{\mf F(A)}$ is a star operation of finite type and, since $A$ is, in particular, a Pr\"ufer domain, $\sharp|_{\mf F(A)}$ is the identity, by the final part of Example \ref{sect:background-semistar}(\ref{1.1d}). Keeping in mind that $\sharp$ is of finite type it follows that, for any $F\in \inssubmod(A)$, we have
$$
F^\sharp=\bigcup\{G^\sharp:G\subseteq F, G\in \mf f(A) \}=\bigcup\{G:G\subseteq F, G\in \mf f(A) \}=F,
$$
i.e., $\sharp$ is the identity semistar operation on $A$, against the fact that $\sharp\geq \star>d$. Thus the only case that may occur is $A^\sharp=K$. In this case, it is easy to infer that $\sharp=\wedge_{\{K  \}}$. 
 Since $\star\neq\wedge_{\{K\}}$, $\star$ is not the infimum of a family of semistar operations of finite type.
\end{ex}

\begin{prop}
Let $A$ be an integral domain and let $\star$ be a semifinite semistar operation on $A$. Then $\widetilde \star=s_{\ad^{\rm i}(\qspec^\star(A))}$.
\end{prop}
\begin{proof}
Since $Y:=\ad^{\rm  i}(\qspec^\star(A))$ is compact (Section \ref{inv-rem}), we have $s_Y=\widetilde{s_{Y}}$. Keeping in mind that $\widetilde \star=s_{\qmax^{\star_f}(A)}=\widetilde{s_{\qmax^{\star_f}(A)}}$, by Proposition \ref{weakly-eq} it is enough to show that $\ad^{\rm i}(Y)=\ad^{\rm i}(\qmax^{\star_f}(A))$.  Let $\f a$ be a finitely generated ideal of $A$ and set $D(\f a):=\spec(A)-V(\f a)$. Since $\star$ is semifinite we have
$$
\qmax^{\star_f}(A)\subseteq D(\f a)\iff \f a^\star\cap A=A\iff \qspec^\star(A)\subseteq D(\f a).
$$
By Proposition \ref{weakly-eq}, the conclusion is now clear. 
\end{proof}
The following corollary is now immediate.
\begin{cor}
Let $A$ be an integral domain and let $\star_1,\star_2$ be semifinite semistar operations on $A$. Then $\widetilde{\star_1}=\widetilde{\star_2}$ if and only if $\ad^{\rm i}(\qspec^{\star_1}(A))=\ad^{\rm i}(\qspec^{\star_2}(A))$. In particular, $\tilde{\star}=d$ if and only if $\qspec^\star(A)$ is dense in $\spec(A)$ with respect to the inverse topology.
\end{cor}
\section*{Acknowledgments}
The authors would like to thank the referee for his/her careful reading of the paper and his/her comments and suggestions. 

\end{document}